\documentclass[12pt,a4paper,oneside]{amsart}

\usepackage[utf8]{inputenc}
\usepackage[english]{babel} 
\usepackage{amsthm} 
\usepackage{amsmath}  
\usepackage{amssymb}
\usepackage[all]{xy}

\makeatletter
\newtheorem*{rep@theorem}{\rep@title}
\newcommand{\newreptheorem}[2]{%
\newenvironment{rep#1}[1]{%
 \def\rep@title{#2 \ref{##1}}%
 \begin{rep@theorem}}%
 {\end{rep@theorem}}}
\makeatother

\newtheorem{theorem}{Theorem}[section]
\newreptheorem{theorem}{Theorem}
\newtheorem{proposition}[theorem]{Proposition}
\newtheorem*{proposition*}{Proposition}
\newtheorem{lemma}[theorem]{Lemma}
\newtheorem{corollary}{Corollary}[theorem]
\theoremstyle{definition}
\newtheorem{definition}[theorem]{Definition}
\newtheorem*{definition*}{Definition}
\newtheorem{example}{Example}[section]
\theoremstyle{remark}
\newtheorem{remark}{Remark}[section]

\title[On the cohomology of surfaces isogenous to a product]{On the cohomology of regular surfaces isogenous to a product of curves with $\chi(\mathcal{O}_S)=2$}
\date{\today}

\author{Matteo A. Bonfanti}
\date{}
\address{Dipartimento di Matematica, Universit\`a degli Studi di Milano, Via Saldini 50, 20133 Milano, Italia}
\email{matteoalfonso.bonfanti@gmail.com}



\begin{document}

%%-->Head
\begin{abstract}
Let $S$ be a surface isogenous to a product of curves of unmixed type.  After presenting several results useful to study the cohomology of $S$ we prove a structure theorem for the cohomology of regular surfaces isogenous to a product of unmixed type with $\chi (\mathcal{O}_S)=2$.
In particular we found two families of surfaces of general type with maximal Picard number.
\end{abstract}

\maketitle

\section*{Introduction} 
Surfaces isogenous to a product of curves have been introduced by Catanese in \cite{Ca00}.  Starting from that paper they have been studied extensively, in particular in the last years.
They provide an easy way to construct surfaces of general type with fixed geometrical invariants.  

Moreover surfaces isogenous to a product are in correspondence with combinatorial structures that a finite group can admit.  
Via this correspondence several authors have classified these surfaces, as in \cite{BCG06}, \cite{CP09}, \cite{Pe10}, \cite{Gl11}.

In this paper we study the cohomology of surfaces isogenous to a product using algebraic methods, in particular group representation theory.  The guiding idea behind is that the cohomology of a surface $S\cong \frac{C\times D}{G}$ is completely determined by the action of the group $G$.  Although our construction is quite general we apply our result to a specific class in order to prove the following:
\begin{reptheorem}{main+}
Let $S$ be a regular surface isogenous to a higher product of unmixed type with $\chi(\mathcal{O}_S)=2$.  Then there exist two elliptic curves $E_C$ and $E_D$ such that $H^2(S,\mathbb{Q})\cong H^2(E_C\times E_D, \mathbb{Q})$ as rational Hodge structures.
\end{reptheorem}

This paper is organized as follows: in the first section we recall all the requires definitions and results; in the second one we study the cohomology of surfaces isogenous to a product of unmixed type, and in particular we focus on the case of regular surfaces with $\chi(\mathcal{O}_S)=2$.  In the third section we study in detail some special surfaces and in the fourth one we present our main result, together with an important observation about the Picard number of the surfaces we studied.

\subsection*{Notation and conventions}
In this paper with curve or surface we mean a complex, smooth projective manifold of complex dimension $1$ or $2$ respectively.  For a given surface $S$ we denote by $\chi(\mathcal{O}_S)$ the holomorphic Euler characteristic, by $e(S)$ the topological Euler characteristic and by $\rho(S)$ the Picard number of $S$.
The invariant $q(S)=h^{1,0}(S)$ is called irregularity: a regular surface $S$ is a surface with $q(S)=0$.

We use also standard notation in group theory: $\mathbb{Z}_n=\mathbb{Z}/n\mathbb{Z}$ is the cyclic group of order $n$; $\mathcal{S}_n,\,\mathcal{A}_n$ and $\mathcal{D}_n$ are respectively the symmetric, the alternating and the dihedral group on $n$ elements.

\subsection*{Acknowledgements}
The author would like to thank his advisor Bert van Geemen for introducing him to the subject.

%%-->Body
%%-->Section 1
\section{Preliminaries and basic results}
In this first section we recall all the definitions and results we need in this paper.  In particular in the Section \ref{GAD} we study in detail the group algebra decomposition.

%%-->Section 1.1
\subsection{Surfaces isogenous to a product}
\begin{definition}
A smooth surface $S$ is said to be isogenous to a product (of curves) if it is isomorphic to a quotient $\frac{C\times D}{G}$ where $C$ and $D$ are curves of genus at least one and $G$ is a finite group acting freely on $C\times D$.\\
If the genus of both curves is greater or equal than two $S$ is said to be isogenous to a higher product.
\end{definition}
\noindent

Let $S\cong \frac{C\times D}{G}$ be a surface isogenous to a product.  The group $G$ is identified with a subgroup of $Aut(C\times D)$ via the group action.  We set 
\begin{equation*}
G^0:=G\cap\left(Aut(C)\times Aut(D)\right).
\end{equation*}
The group $Aut(C)\times Aut(D)$ is a normal subgroup of $Aut(C\times D)$ of index one or two, thus or $G=G^0$ or $[G:G^0]=2$.  In particular an element in the subgroup $G^0$ acts on each curve and diagonally on the product, conversely an element $g\in G$ but not in $G^0$ acts on the product interchangig factors.

\begin{definition}
Let $S$ be a surface isogenous to a product.  Then $\frac{C\times D}{G}$ is a minimal realization of $S$ if $S\cong \frac{C\times D}{G}$ and $G^0$ acts faithfully on both curves.
\end{definition}

\begin{proposition}[\cite{Ca00}, Proposition 3.13]
Let $S$ be a surface isogenous to a higher product.  Then a minimal realization exists and it is unique.
\end{proposition}

From now on whenever we refer to a surface $S$ isogenous to a higher product we will always assume that it is given by its minimal realization.

\begin{definition}
Let $S\cong \frac{C\times D}{G}$ be a surface isogenous to a product.  $S$ is said to be of unmixed type if $G=G^0$, of mixed type otherwise.
\end{definition}

We recall some well known results about surfaces isogenous to a product, in  particular about their invariants.

\begin{proposition}[\cite{Ca00}]
Let $S=\frac{C\times D}{G}$ be a surface isogenous to a higher product.  Then $S$ is minimal surface of general type.
\end{proposition}

\begin{proposition}[\cite{Ca00}, Theorem 3.4]
\label{InvariantSurface}
Let $S\cong\frac{C\times D}{G}$ be a surface isogenous to a product.  Then the following equalities hold:
\begin{itemize}
\item $\chi(\mathcal{O}_S)=\frac{(g(C)-1)(g(D)-1)}{|G|}$;
\item $e(S)=4\chi(\mathcal{O}_S)=\frac{4(g(C)-1)(g(D)-1)}{|G|}$;
\item $K_S^2=8\chi(\mathcal{O}_S)=\frac{8(g(C)-1)(g(D)-1)}{|G|}$.
\end{itemize}
\end{proposition}

\begin{proposition}\label{q}
Let $S\cong\frac{C\times D}{G}$ be a surface isogenous to a product of unmixed type.  Then
\begin{equation*}
q(S)=g\left(C/G\right)+g\left(D/G\right).
\end{equation*}
\end{proposition}

In this paper we focus our attention on regular surfaces isogenous to a higher product of unmixed type with $\chi(\mathcal{O}_S)=2$:

\begin{proposition}\label{diamond}
Let $S$ be a regular surface isogenous to a higher product with $\chi(\mathcal{O}_S)=2$.  Then the Hodge diamond is fixed:
\begin{equation*}
\begin{array}{ccccc}
&&1&&\\
&0&&0&\\
1&&4&&1\\
&0&&0&\\
&&1&&
\end{array}
\end{equation*}
\end{proposition}
\begin{proof}
By hypothesis we have $h^{1,0}(S)=0$ and $h^{2,0}=\chi(\mathcal{O}_S)-1=1$: we just have to compute $h^{1,1}(S)$.  By Proposition \ref{InvariantSurface} $e(S)=4\chi(\mathcal{O}_S)=8$ and then 
\begin{equation*}
h^{1,1}(S)=e(S)-2+4q(S)-2p_g(S)=4
\end{equation*}
\end{proof}
Regular surfaces isogenous to a higher product of unmix type with $\chi(\mathcal{O}_s)=2$ have been studied and classified in \cite{Gl11}: see section \ref{cohom} for the details.

%%-->Section 1.2
\subsection{Spherical system of generators}
We introduce here the notion of spherical systems of generators and we relate them with ramified coverings of the sphere.
We use the same notation of \cite{BCG06}.

\begin{definition}
Let $G$ be a group and $r\in\mathbb{N}$ with $r\ge 2$. An $r$-tuple $T=[g_1,...,g_r]$ of elements in $G$ is called spherical system of generators of $G$ if $g_1,...,g_r$ is a system of generators of $G$ and we have $g_1\cdot...\cdot g_r=Id_G$.\\
We call $\ell(T):=r$ length of $S$.
\end{definition}
\begin{definition}
Let $A=[m_1,\,...,\,m_r]\in\mathbb{N}^r$ be an $r$-tuple of natural numbers $2\le m_1\le ...\le m_r$. A spherical system of generators $T=[g_1,...,g_r]$ is said to be of type $A=[m_1,...,m_r]$ if there is a permutation $\tau\in\mathcal{S}_r$ such that $ord(g_i)=m_{\tau(i)}$, for $i=1,\,...,\,r$.
\end{definition}
\begin{proposition}\label{RETspherical}
Let $G$ be a finite group and $B=\{b_1,\,...,\,b_r\}\subset\mathbb{P}^1$.  Then there is a correspondence between:
\begin{itemize}
\item Spherical system of generators $T$ of $G$ with length $\ell(T)=r$;
\item Galois covering $f:C\to \mathbb{P}^1$ with branch points $B$.
\end{itemize}
\end{proposition}
\begin{proof}
It follows from the Riemann Existence Theorem as explained in \cite[Section III.3 and III.4]{Mi95}.
\end{proof}

\begin{remark}
The curve $C$ is completely determined by the branch points $B$ and by the spherical system of generators $T$.  In particular the genus can be computed using the Riemann-Hurwitz formula:
\begin{equation*}
g(C)=1-d+\sum_{i=1}^r\frac{d}{2m_i}(m_i-1)
\end{equation*}
where $A=[m_1,\,...,\,m_r]$ is the type of $T$.
\end{remark}
\begin{remark}
The correspondence of Proposition \ref{RETspherical} is not one-to-one: indeed distinct spherical systems of generators could determine the same covering.\\
For example let $T_1=[g_1,\,...,\,g_r]$ be a spherical system of generators  of $G$ of type $A$ and let $h\in G$.  Consider $T_2=[g_1^h,\,...,\,g_r^h]$ where $g^h=h^{-1}gh$: $T_2$ is a spherical system of generators of type $A$ and determines an isomorphic covering.  In particular $T_2$ determines exactly the same covering, not only an isomorphic one, and it corresponds to a different choise of the monodromy representation.
\end{remark}

Let $S=\frac{C\times D}{G}$ be a surface isogenous to a higher product of unmixed type with $q(S)=0$.  Then by Proposition \ref{q} we get two ramified coverings of the sphere $f:C\to\mathbb{P}^1$ and $h:D\to\mathbb{P}^1$.  Notice that, from a topological point of view, the surface $S$ is determined by $f$ and $h$ under the further condition that the group $G$ acts freely on the product $C\times D$.

\begin{definition}
Let $T=[g_1,...,g_r]$ be a spherical system of generators of $G$.  We denote by $\Sigma(T)$ the union of all conjugates of the cyclic subgroups generated by the elements $g_1,...,g_r$:
\begin{equation*}
\Sigma(T):=\Sigma([g_1,...,g_r])=\bigcup_{g\in G}\bigcup_{j=0}^\infty\bigcup_{i=1}^r\{g\cdot g_i^j g^{-1}\}.
\end{equation*}
A pair of spherical systems of generators $(T_1,T_2)$ of $G$ is called disjoint if
\begin{equation*}
\Sigma(T_1)\cap\Sigma(T_2)=\{Id_G\}.
\end{equation*}
\end{definition}

\begin{proposition}\label{disjoint-free}
Let $T_1$ and $T_2$ be two spherical systems of generators of $G$ and let $\pi: C\times D\to \frac{C\times D}{G}$ be the induced covering where $G$ acts on the product via the diagonal action.  Then the following conditions are equivalent:
\begin{itemize}
\item $\pi$ is an étale covering, i.e.\ the action of $G$ is free;
\item $(T_1,T_2)$ is a disjoint pair of spherical systems of generators of $G$. 
\end{itemize}
\end{proposition}
\begin{proof}
We observe that an element $g\in G$ fixes a point in $C$ if and only if $g\in\Sigma(T_1)$ and it fixes a point in $D$ if and only if $g\in\Sigma(T_2)$.  Then $g$ fixes a point in $C\times D$ if and only if $g\in \Sigma(T_1)\cap\Sigma(T_2)$.
\end{proof}

\begin{definition}
An unmixed ramification structure for $G$ is a disjoint pair of spherical system of generators $(T_1,T_2)$ of $G$.\\  
Let $A_1=[m_{(1,1)},...,m_{(1,r_1)}]$ and $A_2=[m_{(2,1)},...,m_{(2,r_2)}]$ be respectively a $r_1$-tuple and a $r_2$-tuple of natural numbers with $2\le m_{(1,1)}\le...\le m_{(1,r_1)}$ and $2\le m_{(2,1)}\le...\le m_{(2,r_2)}$.  We say that the unmixed ramification strucure $(T_1,T_2)$ is of type $(A_1,A_2)$ if $T_1$ is of type $A_1$ and $T_2$ is of type $A_2$.
\end{definition}

Putting together Proposition \ref{RETspherical} and Proposition \ref{disjoint-free} we get a correspondence between unmixed ramification structures and surfaces isogenous to a product of unmixed type.  As already observed, this correspondence is not one-to-one, but it works well in one direction: given an unmixed ramification structure it  is uniquely defined a surface isogenous to a product of unmixed type.

%%-->Section 1.3
\subsection{Irreducible rational representation}
We recall some results about irreducible complex representation.  A full discussion with proofs can be found in \cite{Se77}.

Let $G$ be a finite group of order $N$.  We denote by $\rho_i: G\to GL(V_i)$, $i=1,\,...,\,m$ its irreducible complex representations, where $m$ is the number of conjugacy classes in $G$.  We usually denote by $\rho_1$ the trivial representation.
Given a complex representation $\rho: G\to GL(V)$ we denote by $n_\rho(\rho_i)$ the multiplicity of $\rho_i$ in $\rho$.  Then we get:
\begin{equation*}
\rho=\bigoplus_{i=1}^m n_\rho(\rho_i)\rho_i,
\end{equation*}
Let $\chi_i: G\to \mathbb{C}$ be the character associated to the irreducible complex representation $\rho_i$: the character field $K_i$ is the field $\mathbb{Q}(\chi_i(g))_{g\in G}$.  As $\rho_i(g)\in GL(V)$ has finite order, its eigenvalues are roots of unity, hence $K_i$ is a subfield of $\mathbb{Q}(\xi_N)$ where $\xi_N$ is a primitive $N$-th root of unity.

\begin{proposition}\label{GaloisAction1}
Let $G$ be a finite group of order $N$ and let $\rho_i:G\to GL(V_i)$ be an irreducible complex representation of $G$ with associated character field $K_i$.  For every $\sigma\in Gal(K_i/\mathbb{Q})$ there exists an unique irreducible complex representation $\rho_j:G\to GL(V_j)$ with character $\chi_j=\sigma(\chi_i)$.
Thus for $\sigma, \rho_i$ and $\rho_j$ as above we set $\sigma(\rho_i)=\rho_j$.\\
In the same way we can define an action of the whole group $Gal_N$ on the irreducible complex representations.
\end{proposition}

\begin{definition}
Let $\rho_i:G\to GL(V_i)$ be an irreducible complex representation of $G$ with character field $K_i$.  The dual representation of $\rho_i$ is the irreducible complex representation $\overline{\rho_i}:=\tilde\sigma(\rho_i)$ where $\tilde{\sigma}$ is the complex conjugation.\\
We say that $\rho_i$ is self-dual if $\rho_i=\overline{\rho_i}$ or, equivalently, if $K_i\subseteq\mathbb{R}$.
\end{definition}

The action of $Gal_N$ splits the set of the irreducible complex representations into distinct orbits such that if two irreducible complex representations $\rho_i$ and $\rho_j$ are in the same orbit then $K_i=K_j$.

\begin{proposition}\label{RationalRepresentation}
Let $G$ be a finite group of order $N$ and let $\tau:G\to GL(W)$ be an irreducible rational representation.  Then there is a unique $Gal_N$-orbit of irreducible complex representations
\begin{equation*}
\{\sigma(\rho_i)\}_{\sigma\in Gal(K_i/\mathbb{Q})},\quad \rho_i:G\to GL(V_i)
\end{equation*}  
and a positive integer $s$, called Schur index of $\rho_i$, such that
\begin{equation}\label{Qrap}
\tau_\mathbb{C}:=\tau\otimes_\mathbb{Q}\mathbb{C}=\bigoplus_{\sigma\in Gal(K_i/\mathbb{Q})}s\cdot\sigma(\rho_i).
\end{equation}
Conversely each irreducible complex representation $\rho_i$ determines an irreducible rational representation $\tau:G\to GL(W)$ such that the equality \eqref{Qrap} holds.
\end{proposition}

\begin{corollary}
Let $\rho:G\to GL(V)$ be a self-dual complex representation such that $K_\rho=\mathbb{Q}$.  Then there exists a rational representation $\tau:G\to GL(W)$ and a positive integer $s$ such that $\tau\otimes_\mathbb{Q}\mathbb{C}=s\cdot\rho$.
\end{corollary}

Any rational representation $\tau:G\to GL(W)$ can be decomposed as sum of irreducible rational representations, exactly as it happens for the complex ones.  We will write
\begin{equation*}
\tau=\bigoplus^t_{j=1}n_\tau(\tau_j)\tau_j,
\end{equation*}
where $\tau_j:G\to GL(W_j)$, $j=1,...,t$ are the irreducible rational representations of $G$ and $n_\tau(\tau_j)$ is the multiplicity of $\tau_j$ in $\tau$.  As in the complex case, we denote by $\tau_1$ the trivial representation.

\begin{example}\label{quaternion}
Consider the quaternion group $Q_8$:
\begin{equation*}
Q_8=\left\langle -1,\,i,\,j,\,k|\,(-1)^2=1,\,i^2=j^2=k^2=ijk=-1\right\rangle.
\end{equation*}
This is the smallest group with an irreducible representation with Schur index different from one.  The character table of $Q_8$ is:
\begin{equation*}
\begin{array}{c|ccccc}
&1&-1&\pm i&\pm j&\pm k\\
\hline
\chi_1&1&1&1&1&1\\
\chi_2&1&1&1&-1&-1\\
\chi_3&1&1&-1&-1&1\\
\chi_4&1&1&-1&1&-1\\
\chi_5&2&-2&0&0&0
\end{array}
\end{equation*}

In this case any irreducible complex representation $\rho_i$, $i=1,...,5,$ has character field $K_i=\mathbb{Q}$ and then defines a different Galois orbit.
So $G$ has $5$ irreducible rational representations $\tau_j$, $j=1,...,5$.
The Schur index of the first four representations has to be one, because the Schur index divides the dimension of the representation.  Conversely $\rho_5$ has Schur index two.  We can construct the representation $\tau_5$ with $\tau_5\otimes\mathbb{C}=2\rho_5$ setting:
\begin{equation*}\label{repqua}
\tau(i)=\begin{pmatrix} 
0 & -1 & 0 & 0 \\ 
1 & 0 & 0 & 0 \\ 
0 & 0 & 0 & -1\\
0 & 0 & 1 & 0
\end{pmatrix},\qquad
\tau(j)=\begin{pmatrix} 
0 & 0 & -1 & 0 \\ 
0 & 0 & 0 & 1 \\ 
1 & 0 & 0 & 0\\
0 & -1 & 0 & 0
\end{pmatrix}.
\end{equation*}
\end{example}

%%-->Section 1.4
\subsection{Group Algebra decomposition}\label{GAD}
We describe here the so called group algebra decomposition.  The main idea is the following: let $\tau:G\to GL(W)$ be a rational representantion and let $W$ be a rational Hodge structure such that $\tau(G)\subseteq End_{Hod}(W)$.  Then the action of the group algebra $\mathbb{Q}[G]$ induces a decomposition of $W$ into Hodge subrepresentations.  This result is well known in the contest of complex tori (see \cite[Section 13.4]{BL04}): following the same arguments we prove it for Hodge structures.

Let $G$ be a finite group with irreducible complex representations $\rho_i:G\to GL(V_i)$, $i=1,\,...,\,m$, as in the previous section.
Consider the following elements in $\mathbb{C}[G]$:
\begin{equation*}
p_i=\frac{dim(V_i)}{\#G}\sum_{g\in G}\chi_i(g)g,
\end{equation*}
where $\chi_i$ is the character of $\rho_i$.
These elements $p_1,\,...,\,p_m$ are central idempotents in the group algebra $\mathbb{C}[G]$, i.e.\ $p_i^2=p_i$ and $p_ig=gp_i$ for all $g\in G$.  Moreover we get:
\begin{equation}\label{eqc}
\tilde\rho_i(p_j)=\begin{cases} Id_{V_i}&\mbox{if }i=j,\\
0&\mbox{if } i\ne j.
\end{cases}
\end{equation}

Let us consider the group algebra $\mathbb{Q}(\xi_N)[G]$ where $N$ is the order of $G$ and notice that $p_i\in\mathbb{Q}(\xi_N)[G]$ for all $i=1,...,m$.
There is a natural action of the Galois group $Gal_N$ on $\mathbb{Q}(\xi_N)[G]$ defined by
\begin{equation*}
\sigma\left(\sum a_jg_j\right)=\sum\sigma(a_j)g_j
\end{equation*}
where $\sigma\in Gal_N$.\\
This action agrees with the action defined in Proposition \ref{GaloisAction1}: $\sigma(\rho_i)=\rho_j$ if and only if $\sigma(p_i)=p_j$.

Let $\tau_j:G\to GL(W_j)$ be an irreducible rational representation.  By Proposition \ref{RationalRepresentation} there exists an irreducible complex representation $\rho_i:G\to GL(V_i)$ such that 
\begin{equation*}
\tau_j=\bigoplus_{\sigma\in Gal(K_i/\mathbb{Q})}s\cdot \sigma(\rho_i).
\end{equation*}
We define 
\begin{equation*}
q_j=\sum_{\sigma\in Gal(K_i/\mathbb{Q})}\sigma(p_i).
\end{equation*}
\begin{proposition}
Let $G$ be a finite group and let $\tau_j:G\to GL(W_j)$, $j=1,\,...,\,t$ be its irreducible rational representations.  Then $q_j\in\mathbb{Q}[G]$ for all $j=1,\,...,\,t$ and 
\begin{equation}\label{eqq}
\tilde\tau_i(q_j)=\begin{cases} Id_{W_i}&\mbox{if }i=j,\\
0&\mbox{if } i\ne j.
\end{cases}
\end{equation}
\end{proposition}
\begin{proof}
By definition $q_j\in\mathbb{Q}(\xi_N)[G]$.  For all $g\in G$ the coefficient of $g$ in $q_j$ is given by the equation
\begin{equation*}
c_g:=\frac{dim(V_i)}{\#G}\sum_{\sigma\in Gal(K_i/\mathbb{Q})}\sigma(\chi_i(g))
\end{equation*}
By hypothesis $\chi_i(g)\in K_i$ and then $c_g\in\mathbb{Q}$ for all $g\in G$.\\
In order to prove equation \eqref{eqq} we have to complexify it and compare with equation \eqref{eqc}.
\end{proof}
 
\begin{corollary}
Let $\tau:G\to GL(W)$ be a rational representation.  We define $A_j=Im\{\tilde{\tau}(q_j):W\to W\}$.  Then
\begin{itemize}
\item $A_j$ is a rational subrepresentation and $\tau|_{A_j}=m_\tau(\tau_j)\tau_j$;
\item $W=\oplus^t_{j=1}A_j$.
\end{itemize}
\end{corollary}

\begin{definition}
Let $\tau:G\to GL(W)$ be a rational representation.  We call $A_j$ the isotypical component related to the representation $\tau_j$ and we call $W=\oplus^t_{j=1}A_j$ isotypical decomposition of $\tau$.
\end{definition}

Now we need a classical result of representation theory about the group algebra $\mathbb{C}[G]$:
\begin{proposition}\label{isoC}
Let $G$ be a finite group and $\rho_i:G\to GL(V_i)$ $i=1,...,m$ its irreducible complex representations.
We set $\rho=\oplus_{i=1}^m\rho_i:G\to GL(V)$.  Then $\tilde\rho: \mathbb{C}[G]\to\oplus_{i=1}^mEnd(V_i)$ is an algebra isomorphism.
\end{proposition}

\begin{remark}
A similar result can not hold in general for rational representations, since the algebras $\mathbb{Q}[G]$ and $\oplus^t_{j=1}End(W_j)$ need not have the same dimension.
\end{remark}

In order to avoid indices we work on a single irreducible rational representation $\tau:G\to GL(W)$.  Consider $\mathbb{D}:=End_G(W)$, the algebra of $G$-equivariant maps on $W$:
\begin{equation*}
\mathbb{D}=End_G(W)=\{f\in End(W): \tau(g)f=f\tau(g)\,\forall g\in G\}.
\end{equation*}
The kernel of any element $f\in \mathbb{D}$ is a subrepresentation of $W$, hence, as $W$ is irreducible, all $f\in \mathbb{D}$ must be isomorphisms of $W$ and then $\mathbb{D}$ is a skew-field (or a division algebra).
We consider $W$ as a left vector space over $\mathbb{D}$, then choosing a basis we get:
\begin{equation*}
W\cong\mathbb{D}^k,
\end{equation*}
where $k=\dim_\mathbb{D}(W)$.\\
Suppose $\tau_\mathbb{C}=\oplus_{\sigma\in Gal(K_i)}s\cdot\sigma(\rho_i)$, where $\rho_i:G\to GL(V_i)$ is an irreducible complex representation and so
\begin{equation*}
End_G(W_\mathbb{C})=\oplus_{\sigma\in Gal(K_i)}End_G(V_i^{\oplus s}).
\end{equation*}
 Then:
\begin{equation*}
\begin{split}
&\dim_\mathbb{Q}W=\dim_\mathbb{C}W_\mathbb{C}=s\cdot\dim_\mathbb{C}(V_i)\cdot [K_i:\mathbb{Q}],\\
&\dim_\mathbb{Q}\mathbb{D}=\dim_\mathbb{C}\mathbb{D}_\mathbb{C}=[K_i:\mathbb{Q}]\cdot \dim(End_G(V_i^{\oplus s}))=[K_i:\mathbb{Q}]\cdot s^2,\\
&\dim_\mathbb{D}W=k=\frac{[K_i:\mathbb{Q}]\dim_\mathbb{C}(V_i)\cdot s}{[K_i:\mathbb{Q}]\cdot s^2}=\frac{dim_\mathbb{C}(V_i)}{s}.
\end{split}
\end{equation*}
Recall that the Schur index $s$ is always a divisor of the dimension of the representation and so $k\in\mathbb{N}$.
By definition of $\mathbb{D}$, $\tau(g)$ commutes with $\mathbb{D}$ for all $g\in G$ and so the image of $\tilde{\tau}$ lies in $End_\mathbb{D}(W)$.  Moreover we observe that
\begin{equation}\label{dimensionQ}
\begin{split}
\dim_\mathbb{Q}(End_\mathbb{D}(W))=&\dim_\mathbb{Q}\mathbb{D}\cdot\dim_\mathbb{D}End_\mathbb{D}(W)=\\
=&(\dim_\mathbb{C}(V_i))^2\cdot [K_i:\mathbb{Q}].
\end{split}
\end{equation}
\begin{proposition}
Let $G$ be a finite group and $\tau_j:G\to GL(W_j)$, $j=1,...,t$, its irreducible rational representations.
We set $\mathbb{D}_j=End_G(W_j)$ and $\tau=\oplus_{j=1}^t\tau_j:G\to GL(W)$.  Then $\tilde\tau: \mathbb{Q}[G]\to\oplus_{j=1}^tEnd_{\mathbb{D}_j}(W_j)$ is an algebra isomorphism.
\end{proposition}
\begin{proof}
From Proposition \ref{isoC} we get the injectivity.  Then it is enough to prove that the two algebras have the same dimension.  Of course $\dim\mathbb{Q}[G]=\#G$.  Now from equation \eqref{dimensionQ} we get
\begin{equation*}
\dim_\mathbb{Q}\left(\oplus^t_{j=1}End_{\mathbb{D}_j}(W_j)\right)=\oplus_{i=1}^m(\dim_\mathbb{C}(V_i))^2=\#G.
\end{equation*}
\end{proof}

By choosing a $\mathbb{D}$-basis of $W$ we identiy $End_\mathbb{D}(W)$ with the algebra $Mat(k,\mathbb{D})$ of matrices $k\times k$  with coefficients in $\mathbb{D}$.
In particular in $Mat(k,\mathbb{D})$ we have matrices $E_i$ with $1$ at $(i,i)$ and zero elsewhere.
Then by the proposition above we are able to find $k$ idempotents $w_1,\,...,\,w_k$ in $\mathbb{Q}[G]$ such that $\tilde\tau(w_i)=E_i$.
\begin{remark}
This elements $w_1,\,...,\,w_k$ are not unique, since they depend on the choice of a $\mathbb{D}$-basis.
\end{remark}
This construction holds for all the irreducible rational representations.  Given an irreducible rational representation $\tau_j:G\to GL(W_j)$ we denote by $w_{j,1},\,...,\,w_{j,k_j}$ idempotent elements of $\mathbb{Q}[G]$ constructed as above.
\begin{proposition}\label{isogenousdecomposition}
Let $\tau:G\to GL(W)$ be a rational representation and let $A_1,\,...,\,A_t$ be the isotypical components related to the irreducible rational representations of $G$.  For all $j \in 1,..., t$ we define $B_j=Im\{\tilde{\tau}(w_{j,1}):W\to W\}$.  Then $A_j\cong B_j^{\oplus k_j}$ for all $j$, $j=1,\,...,\,t$, $k_j=dim_{\mathbb{D}_j}W_j$.
\end{proposition}
\begin{proof}
By construction $w_{j,1}+\,...+\,w_{j,k_j}=q_j$ for all $j=1,\,...,\,t$.  Since $\tilde{\tau}(q_j)$ acts as the identity on $A_j$ we get a decomposition:
\begin{equation*}
A_j=Im\{\tilde{t}(w_{j,1})\}\oplus\,...\oplus\,Im\{\tilde{t}(w_{j,k_j})\}.
\end{equation*} 
Fix a $\mathbb{D}_j$-basis of $W_j$ and consider in $End_{\mathbb{D}_j}(W_j)\cong Mat(k_j,\mathbb{D}_j)$ the matrices $M_i$ with $1$ at $(i,1)$ and zero elsewhere.  These matrices provide isomorphisms between $B_j=Im\{\tilde{t}(w_{j,1})\}$ and $Im\{\tilde{t}(w_{j,i})\}$ for all $i=2,\,...,\,k_j$.
\end{proof}

\begin{definition}
Let $\tau:G\to GL(W)$ be a rational representation.  We call $B_j$ the isogenous component related to the representation $\tau_j$ and we call $W\cong\oplus^t_{j=1}B_j^{\oplus k_j}$ the group algebra decomposition of $\tau$.
\end{definition}

\begin{remark}
Unlike the isotypical components $A_j$, the isogenous components $B_j$ are not $G$-subrepresentations.  Indeed, as observed in the proof of the Proposition \ref{isogenousdecomposition}, the group algebra $\mathbb{Q}[G]$ interchanges the isogenous components.
\end{remark}

Now that we have defined the group algebra decomposition we relate it with the Hodge structures.  First of all we recall the following:

\begin{lemma}[\cite{Vo02}, Section 7.3.1]\label{imm}
Let $W$ be a rational Hodge strucuture and let $\phi\in End_{Hod}(W)$.  Then $Im(\phi)$ is a rational Hodge substrucure.
\end{lemma}

\begin{proposition}
Let $(W,h)$ be a rational Hodge structure, $G$ a finite group and let $\tau:G\to GL(W)$ be a rational representation such that $\tau(G)\subset End_{Hod}(W)$.  Then the isotypical and isogenous component of $\tau$ are Hodge substructures.
\end{proposition}
\begin{proof}
We have defined the isotypical and isogenous components of a given representation $\tau:G\to GL(W)$ as the images of opportune elements in $\mathbb{Q}[G]$.  Notice that $\tau(G)\subseteq End_{Hod}(W)$ implies $\tilde{\tau}(\mathbb{Q}[G])\subseteq End_{Hod}(W)$.  Now we apply Lemma \ref{imm}
\end{proof}

We conclude this section with the following lemma:
\begin{lemma}\label{pari}
Let $G$ be a finite group and $\rho_i:G\to GL(V_i)$ its irreducible complex representations.  Let $(W,h)$ be a rational Hodge structure of weight $1$ and $\tau:G\to GL(W)$ a rational representation such that $\tau(G)\subset End_{Hod}(W)$.  Consider the induced complex representations $\tau_\mathbb{C}:G\to GL(W_\mathbb{C})$ and $\rho=\tau |_{W^{1,0}}:G\to GL(W^{1,0})$.  Then:
\begin{itemize}
\item $n_{\tau_\mathbb{C}}(\rho_i)=n_{\rho}(\rho_i)+n_\rho(\overline{\rho_i})$;
\item if $\rho_i$ is self-dual $n_{\tau_\mathbb{C}}(\rho_i)$ is even.
\end{itemize}
\end{lemma}
\begin{proof}
The subspaces $W^{1,0}$ and $W^{0,1}$ are subrepresentations of $W_\mathbb{C}$.  It follows that if $\tau_\mathbb{C}|_{W^{1,0}}=\rho$ then $\tau_\mathbb{C}|_{W^{0,1}}=\overline{\rho}$, i.e.\ $\tau_\mathbb{C}=\rho\oplus\overline{\rho}.$
Hence the following equalities hold:
\begin{equation*}
\begin{split}
&n_{\tau_\mathbb{C}}(\rho_i)=n_\rho(\rho_i)+n_{\overline{\rho}}(\rho_i),\\
&n_{\overline{\rho}}(\rho_i)=n_\rho(\overline{\rho_i}).
\end{split}
\end{equation*}
In particular if $\rho_i$ is self-dual we get $n_{\tau_\mathbb{C}}(\rho_i)=2n_\rho(\rho_i)$.
\end{proof}

%%-->Section 1.5
\subsection{Broughton formula}

Let $C$ be a smooth curve of genus $g(C)$ and let $G$ be a finite group of automorphisms of $C$.  We will denote by $\varphi$ the natural action induced by $G$ on the first cohomology group $H^1(C,\mathbb{C})$.
Let assume $C/G\cong\mathbb{P}^1$ and let $T=[g_1,\,...,\,g_r]$ be the spherical system of generators associated to the ramified covering $f:C\to C/G\cong\mathbb{P}^1$.
\begin{proposition}[\cite{Br87}]\label{Broughton}
Let $\varphi=\oplus_{i=1}^mn_{\varphi}(\rho_i)\rho_i$ be the decomposition of $\varphi$ into irreducible complex representations.  Then, with the notation as above we have:
\begin{itemize}
\item $n_\varphi(\rho_1)=\langle\varphi,\rho_1\rangle =0$,
\item $n_\varphi(\rho_i)=\langle\varphi,\rho_i\rangle =\chi_i(1)(r-2)-\sum_{j=1}^rl_{g_j}(\rho_i)$,
\end{itemize}
where $\chi_i$ are the characters of the irreducible complex representations $\rho_i:G\to GL(V_i)$ of $G$, 
$\rho_1$ is the trivial representation, $r=\ell(T)$ is the length of $T$ and $l_{g_j}(\rho_i)$ is the multiplicity of the trivial character in the restriction of $\rho_i$ to $\langle g_j\rangle$.
\end{proposition}
\begin{remark}
The same computations can be done using the Lefschetz fixed point formula (see \cite[Chapter 3.4]{GH94}).  However, since we are interested only in the first cohomology groups of curves, Broughton's formula makes calculations faster and easier.
\end{remark}
The group $G$ induces an action not only on the complex (or real) cohomology, but also in the rational one.  These actions are connected since $H^1(C,\mathbb{C})=H^1(C,\mathbb{Q})\otimes\mathbb{C}$.  We will denote both actions with $\varphi$.\\
Applying together Proposition \ref{RationalRepresentation} and Proposition \ref{Broughton} we can compute the decomposition of $\varphi$ into irreducible rational representations.

Notice that here we are exactly in the situation described in Proposition \ref{pari}: the finite group $G$ acts on $H^1(C,\mathbb{Q})$ that is a rational Hodge structure of weight $1$.  Moreover, since $G$ acts holomorphically on $C$, the action on the cohomology preserves the Hodge decomposition and then 
\begin{equation*}
\varphi(G)\subset End_{Hod}(H^1(C,\mathbb{Q})).
\end{equation*}
\begin{example}\label{z32}
Let $G$ be the abelian group $(\mathbb{Z}_3)^2:=(\mathbb{Z}/3\mathbb{Z})^2$.  Consider the unmixed ramification structure $(T_1,T_2)$ for $G$:
\begin{equation*}
\begin{split}
T_1=&[(1,1),(2,1),(1,1),(1,2),(1,1)],\\
T_2=&[(0,2),(0,1),(1,0),(2,0)],
\end{split}
\end{equation*}
of type $([3^5],\,[3^4])$.  We denote by $f$ and $h$ the  corresponding ramified coverings of $\mathbb{P}^1$:
\begin{equation*}
\begin{split}
f:C&\to C/G\cong\mathbb{P}^1,\\
h:D&\to D/G\cong\mathbb{P}^1,
\end{split}
\end{equation*}
where $C$ and $D$ have genus $7$ and $4$ respectively.\\
The character table of $G$ is 
\begin{equation*}
\begin{array}{c|ccccccccc}
&Id&(1,0)&(2,0)&(0,1)&(0,2)&(1,1)&(2,2)&(2,1)&(1,2)\\
\hline
\chi_1&1&1&1&1&1&1&1&1&1\\
\chi_2&1&\xi_3&\xi_3^2&1&1&\xi_3&\xi_3^2&\xi_3^2&\xi_3\\
\chi_3&1&\xi_3^2&\xi_3&1&1&\xi_3^2&\xi_3&\xi_3&\xi_3^2\\
\chi_4&1&1&1&\xi_3&\xi_3^2&\xi_3&\xi_3^2&\xi_3&\xi_3^2\\
\chi_5&1&\xi_3&\xi_3^2&\xi_3&\xi_3^2&\xi_3^2&\xi_3&1&1\\
\chi_6&1&\xi_3^2&\xi_3&\xi_3&\xi_3^2&1&1&\xi_3^2&\xi_3\\
\chi_7&1&1&1&\xi_3^2&\xi_3&\xi_3^2&\xi_3&\xi_3^2&\xi_3\\
\chi_8&1&\xi_3&\xi_3^2&\xi_3^2&\xi_3&1&1&\xi_3&\xi_3^2\\
\chi_9&1&\xi_3^2&\xi_3&\xi_3^2&\xi_3&\xi_3&\xi_3^2&1&1\\
\end{array}
\end{equation*}
By Proposition \ref{RationalRepresentation} $G$ has $5$ irreducible $\mathbb{Q}$-representations $\tau_1,\,...,\,\tau_5$ with:
\begin{equation*}
\begin{split}
\tau_1\otimes_\mathbb{Q}\mathbb{C}=&\rho_1,\\
\tau_2\otimes_\mathbb{Q}\mathbb{C}=&\rho_2\oplus\rho_3,\\
\tau_3\otimes_\mathbb{Q}\mathbb{C}=&\rho_4\oplus\rho_7,\\
\tau_4\otimes_\mathbb{Q}\mathbb{C}=&\rho_5\oplus\rho_9,\\
\tau_5\otimes_\mathbb{Q}\mathbb{C}=&\rho_6\oplus\rho_8.\\
\end{split}
\end{equation*}
We apply the Broughton formula (Proposition \ref{Broughton}) to compute the decomposition of the representaion of the group $G$ on $H^1(C,\mathbb{C})$ and $H^1(D,\mathbb{C})$.  We get 
\begin{equation*}
\begin{array}{c|ccccccccc}
&\rho_1&\rho_2&\rho_3&\rho_4&\rho_5&\rho_6&\rho_7&\rho_8&\rho_9\\
\hline
\varphi_C&0&3&3&3&1&0&3&0&1\\
\varphi_D&0&0&0&0&2&2&0&2&2
\end{array}
\end{equation*}
for the complex cohomology groups $H^1(C,\mathbb{C})$, $H^1(D,\mathbb{C})$ and 
\begin{equation*}
\begin{array}{c|ccccc}
&\tau_1&\tau_2&\tau_3&\tau_4&\tau_5\\
\hline
\varphi_C&0&3&3&1&0\\
\varphi_D&0&0&0&2&2
\end{array}
\end{equation*}
for the rational cohomology groups $H^1(C,\mathbb{Q})$, $H^1(D,\mathbb{Q})$.
\end{example}

%%-->Section 2
\section{On the Cohomology}\label{cohom}

The cohomology of surfaces isogenous to a product has been studied in \cite{CLZ13} and in \cite{CL13}.  In these papers the authors focused on the complex cohomology and they study the corresponding Albanese variety.  We follow here a completely different approach.

Let $S=\frac{C\times D}{G}$ be a surface isogenous to a higher product of unmixed type.  Then the second cohomology of $S$ depends on the cohomology of $C$ and $D$ and on the action of $G$.  First of all we need a topological lemma:

\begin{lemma}[\cite{Ha02}, Proposition 3G.1] \label{CoveringCohomology}
Let $\pi:\tilde{X}\to X$ be a (topological) covering space of degree $N$ defined by an action of a group $G$ on $\tilde{X}$.  Then with coefficients in a field $F$ whose characteristic is $0$ or a prime not dividing $n$, the map $\pi^*:H^k(X,F)\to H^k(\tilde{X},F)$ is injective with image the subgroup $H^k(\tilde{X},F)^G$.
\end{lemma}

\begin{proposition}\label{decompositionsecond}
Let $S=\frac{C\times D}{G}$ be a surface isogenous to a higher product of unmixed type.  Then the second cohomology group of $S$ is given by $H^2(S,\mathbb{Q})\cong U\oplus Z$, where
\begin{equation*}
\begin{split}
U:&=\left(H^2(C,\mathbb{Q})\otimes H^0(D,\mathbb{Q})\right)\oplus\left(H^0(C,\mathbb{Q})\otimes H^2(D,\mathbb{Q})\right),\\
Z:&=\left(H^1(C,\mathbb{Q})\otimes H^1(D,\mathbb{Q})\right)^G.
\end{split}
\end{equation*}
\end{proposition}
\begin{proof}
We compute the second cohomology of $C\times D$ with the Künneth formula (see \cite[Theorem 3.16]{Ha02}) and we apply Lemma \ref{CoveringCohomology}.  Since $G$ acts trivially on the zero cohomology and on the second cohomology of the curves $C$ and $D$ we get the result.
\end{proof}

\begin{remark}
Consider $H^2(S,\mathbb{Q})$ as rational Hodge structure of weight $2$.  Then $U,Z\le H^2(S,\mathbb{Q})$ are Hodge substructures.  In particular the subspace $U$ has dimension $2$, and $U\otimes_\mathbb{Q}\mathbb{C}\le H^{1,1}(S)$.  Then, as rational Hodge structure, it is isomorphic to the Tate structure $\mathbb{Q}^2(-1)$.
It follows that $H^2(S,\mathbb{Q})$ is determined, as Hodge structure, by $Z$.
\end{remark}

We recall a classical result of representation theory:
\begin{lemma}\label{productcompl}
Let $G$ be a finite group and let $\rho_i:G\to GL(V_i)$ $i=1,\,...,\,m$ be its irreducible complex representations, where $\rho_1$ is the trivial representation.  Then
\begin{equation*}
n_{\rho_i\otimes\rho_j}(\rho_1)=\langle\rho_i\otimes\rho_j,\rho_1\rangle =\begin{cases} 
1&\mbox{if }\rho_j=\overline{\rho_i},\\
0&\mbox{otherwise}.
\end{cases}
\end{equation*}
\end{lemma}

We need the corresponding result for rational representation:
\begin{proposition}\label{productrepr}
Let $G$ be a finite group and let $\tau_j:G\to GL(W_j)$ $j=1,\,...,\,t$ be its irreducible rational representations, where $\tau_1$ is the trivial representation. Then the multiplicity of the trivial representation in $\tau_j\otimes \tau_k$ is:
 \begin{equation*}
n_{\tau_j\otimes\tau_k}(\tau_1)=\begin{cases} 
s^2[K_i:\mathbb{Q}]&\mbox{if } j=k,\\
0&\mbox{otherwise},
\end{cases}
\end{equation*}
where $\tau_j\otimes\mathbb{C}=s\bigoplus_{\sigma\in Gal(K_i/\mathbb{Q})}\sigma(\rho_i)$.
\end{proposition}
\begin{proof}
It follows fom Lemma \ref{productcompl} and Proposition \ref{RationalRepresentation}.
\end{proof}
Let $\tau_{j_1}:G\to GL(W_1)$ and $\tau_{j_2}:G\to GL(W_2)$ be two irreducible rational representations of $G$.  The group acts trivially on $(W_{j_1}\otimes W_{j_2})^G$ and then
\begin{equation*}
\dim (W_{j_1}\otimes W_{j_2})^G= n_{\tau_{j_1}\otimes\tau_{j_2}}(\tau_1).
\end{equation*}
In particular $\dim (W_{j_1}\otimes W_{j_2})^G\ne 0$ if and only if $j_1=j_2$, and in this case the dimension is determined by Proposition \ref{productrepr}.

Let $S=\frac{C\times D}{G}$ be a surface isogenous to a higher product of unmixed type and let $\tau_j:G\to GL(W_j)$, $j=1,\,...,\,t$ be the irreducible rational representations of $G$.  Let $\varphi_C:G\to GL(H^1(C,\mathbb{Q}))$ and $\varphi_D:G\to GL(H^1(D,\mathbb{Q}))$ be the actions induced by $G$ on the first cohomology of curves:
\begin{equation*}
\begin{split}
&\varphi_C=n_C(\tau_1)\tau_1\oplus\,...\oplus\,n_C(\tau_t)\tau_t,\\
&\varphi_D=n_D(\tau_1)\tau_1\oplus\,...\oplus\,n_D(\tau_t)\tau_t.
\end{split}
\end{equation*}
Then each irreducible rational representation $\tau_j$ determines a subspace of the rational Hodge structure $Z$ of dimension 
\begin{equation*}
n_C(\tau_j)n_D(\tau_j)n_{\tau_j\otimes\tau_j}(\tau_1).
\end{equation*}
In particular we obtain 
\begin{equation*}
\dim Z=\sum_{j=1}^{t}n_C(\tau_j)n_D(\tau_j)n_{\tau_j\otimes\tau_j}(\tau_1).
\end{equation*}

Let we focus on the case of regular surfaces isogenous to a higher product with $\chi(\mathcal{O}_S)=2$.

\begin{proposition}\label{classification}
Let $S=\frac{C\times D}{G}$ be a regular surface isogenous to a higher product with $\chi(\mathcal{O}_S)=2$.  Then one of the following cases holds:
\begin{itemize}
\item[a)] There exists an absolutely irreducible rational representation $\tau: G\to GL(W)$ such that
\begin{equation*}
\begin{split}
&n_C(\tau)=n_D(\tau)=2,\\
&n_C(\tau_j)\cdot n_D(\tau_j)=0,\quad\forall\,\tau_j\mbox{ different from }\tau.
\end{split}
\end{equation*}
\item[b)] There exists an irreducible rational representation $\tau:G\to GL(W)$ and an irreducible complex representation $\rho:G\to GL(V)$ with $\tau_\mathbb{C}=2\rho$ such that 
\begin{equation*}
\begin{split}
&n_C(\tau)=n_D(\tau)=1,\\
&n_C(\tau_j)\cdot n_D(\tau_j)=0,\quad\forall\,\tau_j\mbox{ different from }\tau.
\end{split}
\end{equation*}
\item[c)] There exists an irreducible rational representation $\tau:G\to GL(W)$ and an irreducible complex representation $\rho:G\to GL(V)$ with $\tau_\mathbb{C}=\rho\oplus\overline{\rho}$ such that 
\begin{equation*}
\begin{split}
&n_C(\tau)=1,\quad n_D(\tau)=2,\\ 
&n_C(\tau_j)\cdot n_D(\tau_j)=0,\quad\forall\,\tau_j\mbox{ different from }\tau.
\end{split}
\end{equation*}
\item[d)] There exist two irreducible rational representations $\tau_{j_1}:G \to GL(W_{j_1})$, $\tau_{j_2}:G\to GL(W_{j_2})$ and two irreducible complex representations $\rho_{i_1}:G\to GL(V_{i_1})$, $\rho_{i_2}:G\to GL(V_{i_2})$ with $\tau_{j_1}\otimes\mathbb{C}=\rho_{i_1}\oplus\overline{\rho_{i_1}}$, $\tau_{j_2}\otimes\mathbb{C}=\rho_{i_2}\oplus\overline{\rho_{i_2}}$ and $j_1\neq j_2$ such that 
\begin{equation*}
\begin{split}
&n_C(\tau_{j_1})=n_C(\tau_{j_2})=n_D(\tau_{j_1})=n_D(\tau_{j_2})=1,\\
&n_C(\tau_j)\cdot n_D(\tau_j)=0,\quad\forall\,\tau_j\mbox{ different from }\tau_{j_1},\,\tau_{j_2}.
\end{split}
\end{equation*}
\end{itemize}
\end{proposition}
\begin{proof}
For a regular surface $S$ isogenous to a higher product with $\chi(\mathcal{O}_S)=2$ we have $\dim Z=4$.
Notice that, for all the irreducible rational representations, the number $n_C(\tau_j)n_D(\tau_j)n_{\tau_j\otimes\tau_j}(\tau_1)$ is even by Lemma \ref{pari}.  Then we can have at most two irreducible rational representations $\tau_j$ such that $n_C(\tau_j)n_D(\tau_j)\ne 0$.
Suppose we have only one: then, again for Lemma \ref{pari}, we have three possibilities:
\begin{itemize}
\item $n_C(\tau_j)=n_D(\tau_j)=2$, $n_{\tau_j\otimes\tau_j}(\tau_1)=1$: this is the case $a$;
\item $n_C(\tau_j)=n_D(\tau_j)=1$, $n_{\tau_j\otimes\tau_j}(\tau_1)=4$: this is the case $b$;
\item $n_C(\tau_j)=2$, $n_D(\tau_j)=1$, $n_{\tau_j\otimes\tau_j}(\tau_1)=2$: this is the case $c$.
\end{itemize}
Suppose now we have contributions from two different irreducible rational representations $\tau_{j_1}$ and $\tau_{j_2}$.  Then
\begin{itemize}
\item $n_C(\tau_{j_i})=n_D(\tau_{j_i})=1$, $n_{\tau_{j_i}\otimes\tau_{j_i}}(\tau_1)=2$ for $i=1,2$: this is the case $d$.
\end{itemize}
\end{proof}

\begin{definition}
Let $S=\frac{C\times D}{G}$ be a regular surface isogenous to a higher product of unmixed type with $\chi(\mathcal{O}_S)=2$.  We say that $S$ is of type $a,\,b,\,c$ or $d$ if the corresponding case of Proposition \ref{classification} holds for $S$.
\end{definition}

As already mentioned, regular surfaces isogenous to a higher product of unmixed type with $\chi(\mathcal{O}_S)=2$ have been classified by Gleissner.  In \cite{Gl11} he proves that only $21$ groups admit an unmixed ramification structure such that the corresponding surface has $\chi(\mathcal{O}_S)=2$ and $q(S)=0$.  In particular $7$ groups admit more than one non-isomorphic structures, and he obtains $32$ families of regular surfaces isogenous to a higher product of unmixed type with $\chi(\mathcal{O}_S)=2$.  A complete list can be found in Table \ref{tabella} while the explicit forms of the unmixed ramification structures can be found in \cite{Gl11}
For all the surfaces in the list we determined if they are of type $a,\,b,\,c$ or $d$.

Let $G$ be one of the $14 $ groups in the following list:
\begin{equation*}
\begin{split}
&(\mathbb{Z}_2)^3\rtimes_\varphi\mathcal{S}_4,\quad
(\mathbb{Z}_2)^4\rtimes_\varphi \mathcal{D}_5,\quad
\mathcal{S}_5,\quad
(\mathbb{Z}_2)^4\rtimes_\psi \mathcal{D}_3,\\
&U(4,2),\quad
\mathcal{A}_5,\quad
\mathcal{S}_4\times\mathbb{Z}_2,\quad 
\mathcal{D}_4\times (\mathbb{Z}_2)^2,\quad (\mathbb{Z}_2)^4\rtimes_\varphi\mathbb{Z}_2,\\
&\mathcal{S}_4,\quad
\mathcal{D}_4\times\mathbb{Z}_2,\quad
(\mathbb{Z}_2)^2\rtimes_\varphi\mathbb{Z}_4,\quad
(\mathbb{Z}_2)^4,\quad
(\mathbb{Z}_2)^3.
\end{split}
\end{equation*}
For all the irreducible complex representations $\rho:G\to GL(V)$ we get $K_\rho\subseteq\mathbb{R}$ and the Schur index of $\rho$ is equal $1$.
Therefore the corresponding surfaces $S$ are of type $a$.

We verified that also the surfaces related to the groups 
\begin{equation*}
PSL(2,\mathbb{F}_7)\times\mathbb{Z}_2,\quad
PSL(2,\mathbb{F}_7),\quad
(\mathbb{Z}_2)^3\rtimes_\varphi\mathcal{D}_4
\end{equation*}
are of type $a$, although these groups admit irreducible complex representations with $K_\rho\not\subseteq\mathbb{R}$.

\begin{example}
As example we study in detail the group $PSL(2,\mathbb{F}_7)$.  $G:=PSL(2,\mathbb{F}_7)$ has $6$ irreducible complex representations $\rho_1,...,\rho_6$ associated to the charaters $\chi_1,...,\,\chi_6$:
\begin{equation*}
\begin{array}{c|cccccc}
&Id&2&3&4&7a&7b\\
\hline
\chi_1&1&1&1&1&1&1\\
\chi_2&3&-1&0&1&\xi&\overline{\xi}\\
\chi_3&3&-1&0&1&\overline{\xi}&\xi\\
\chi_4&6&2&0&0&-1&-1\\
\chi_5&7&-1&1&-1&0&0\\
\chi_6&8&0&-1&0&1&1\\
\end{array}
\end{equation*}
where $\xi=\frac{-1+i\sqrt{7}}{2}$.  By Proposition \ref{RationalRepresentation}, $G$ has only $5$ irreducible rational representations $\tau_1,...,\,\tau_5$.  One has
\begin{equation*}
\begin{split}
\tau_1\otimes_\mathbb{Q}\mathbb{C}=&\rho_1,\\
\tau_2\otimes_\mathbb{Q}\mathbb{C}=&\rho_2\oplus\rho_3,\\
\tau_3\otimes_\mathbb{Q}\mathbb{C}=&\rho_4,\\
\tau_4\otimes_\mathbb{Q}\mathbb{C}=&\rho_5,\\
\tau_5\otimes_\mathbb{Q}\mathbb{C}=&\rho_6.
\end{split}
\end{equation*}
The group $PSL(2,\mathbb{F}_7)$ admits two non-isomorphic unmixed structures $(T_{C_1}, T_{D_1})$ and $(T_{C_2}, T_{D_2})$ of types $([7^3],\,[3^2,4])$ and $([3^2,7],\,[4^3])$ respectively.
Since there is only one conjugacy class of elements of order $3$ and one conjugacy class of elements of order $4$ in $G$ (denoted in the table above with $3$ and $4$), we can apply the Broughton formula to the curves $D_1$ and $D_2$ easily.  We get:
\begin{equation*}
\begin{array}{c|ccccc}
&\tau_1&\tau_2&\tau_3&\tau_4&\tau_5\\
\hline
\varphi_{D_1}&0&0&0&0&2\\
\varphi_{D_2}&0&0&0&4&2
\end{array}
\end{equation*}
So, even if $G$ has two non self-dual representations ($\rho_2$ and $\rho_3$), the surfaces isogenous to a higher product associated to both $(T_{C_1}, T_{D_1})$ and $(T_{C_2}, T_{D_2})$ are of type $a$.
\end{example}

Finally surfaces related to the groups 
\begin{equation*}
G(128,36),\quad
(\mathbb{Z}_2)^4\rtimes_\varphi \mathcal{D}_3,\quad (\mathbb{Z}_2)^3\rtimes_\varphi\mathbb{Z}_4,\quad
(\mathbb{Z}_3)^2,
\end{equation*}
are not of type $a$ and we will study them in the next section.
The complete list, with the corresponding type, is summarized in Table \ref{tabella}, at the end of this section.

\begin{theorem}\label{main}
Let $S=\frac{C\times D}{G}$ be a regular surface isogenous to a higher product of unmixed type with $\chi(\mathcal{O}_S)=2$ and assume that $S$ is of type $a$.
Then there exist two elliptic curves $E_C$ and $E_D$ such that $H^2(S,\mathbb{Q})\cong H^2(E_C\times E_D,\mathbb{Q})$ as rational Hodge structures.
\end{theorem}

\begin{proof}
The proof consists of two steps: in the first one we construct the two elliptic curves $E_C$ and $E_D$; in the second one we prove that $H^2(S,\mathbb{Q})\cong H^2(E_C\times E_D,\mathbb{Q})$.\\
\textbf{Step 1}: By hypothesis there exists an absolutely irreducible rational representation $\tau:G\to GL(W)$ such that $n_C(\tau)=n_D(\tau)=2$; let $\dim W=n$.  We denote by $A_C$ and $A_D$ the isotypical components related to $\tau$ in $H^1(C,\mathbb{Q})$ and $H^1(D,\mathbb{Q})$: $A_C$ and $A_D$ are, at the same time, rational Hodge substructure and $G$-subrepresentations of dimension $2n$ and we obtain
\begin{equation*}
Z\cong \left(A_C\otimes A_D\right)^G,
\end{equation*}
where $Z$ is the Hodge substructure defined in Proposition \ref{decompositionsecond}.  Since $\tau$ is an absolutely irreducible rational representation, the corresponding skew-field $\mathbb{D}$ is simply $\mathbb{Q}$.  So, by Proposition \ref{isogenousdecomposition}, we get $A_C\cong B_C^{\oplus n}$ and $A_D\cong B_D^{\oplus n}$ where $B_C$ and $B_D$ are Hodge substructures, but no longer $G$-subrepresentations, of $A_C$ and $A_D$ with $\dim B_C=\dim B_D=2$.  Via the natural correspondence between complex tori and Hodge structures, there exists two elliptic curves $E_C$ and $E_D$, defined up to isogeny, such that 
\begin{equation*}
B_C\cong H^1(E_C,\mathbb{Q}),\qquad B_D\cong H^1(E_D,\mathbb{Q}),
\end{equation*}
as rational Hodge structures.\\
\textbf{Step 2}:
The Hodge structures of weight two $Z$ and $B_C\otimes B_D$ have the same dimension and the same Hodge numbers; in particular $\dim Z^{2,0}=\dim (B_C\times B_D)^{2,0}=1$.  The action of $G$ provides a Hodge homomorphism $A_C\otimes A_D\to Z$: by restriction we get a map $\psi: B_C\otimes B_D\to Z$.  
Consider the Hodge substructure $Im(\psi)$. We can assume that $\dim Im(\psi)^{2,0}=1$: otherwise we have to change the  choice of $B_C$ and $B_D$ in $A_C$ and $A_D$.
If $\psi$ is an isomorphism we are done.  Otherwise let $k:=\dim Ker(\psi)$.  We have the decompositions:
\begin{equation*}
B_C\otimes B_D\simeq P\oplus Ker(\psi),\qquad Z\simeq Im(\psi)\oplus \mathbb{Q}^k(-1),
\end{equation*}
where $P$ is a Hodge substructure with $\dim P^{2,0}=1$ and $\dim P^{1,1}=2-k$.  The Hodge structures $B_C\otimes B_D$ and $Z$ are isomorphic since
\begin{itemize}
\item $\psi$ defines an isomorphism between $P$ and $Im(\psi)$,
\item $\dim ker(\psi)^{2,0}=0$ and then $ker(\psi)\simeq \mathbb{Q}^k(-1)$.
\end{itemize}
\end{proof}

\begin{table}
\begin{center}
\begin{tabular}{|c|c|c|c|c|c|}
\hline 
$G$& $|G|$ & SGL &$g(C)$ & $g(D)$& type\\
\hline
$PSL(2,\mathbb{F}_7)\times\mathbb{Z}_2$& 336 & $\langle 336,209\rangle$ & 17 & 43 & a\\
$(\mathbb{Z}_2)^3\rtimes_\varphi\mathcal{S}_4$& 192& $\langle 192,995\rangle$ & 49 & 9 & a\\
$PSL(2,\mathbb{F}_7)$& 168 & $\langle 168,42\rangle$& 49 & 8 & a\\
$PSL(2,\mathbb{F}_7)$& 168 & $\langle 168,42\rangle$& 17 & 22 & a\\
$(\mathbb{Z}_2)^4\rtimes_\varphi \mathcal{D}_5$& 160 & $\langle 160,234\rangle$& 5 & 81 & a\\
$G(128,36)$ & 128 & $\langle 128,36\rangle$& 17 & 17 & \textbf{b}\\
$\mathcal{S}_5$ & 120 & $\langle 120,34\rangle$& 9 & 31 & a\\
$(\mathbb{Z}_2)^4\rtimes_\varphi \mathcal{D}_3$ & 96 & $\langle 96,195\rangle$& 5 & 49 & \textbf{c}\\
$(\mathbb{Z}_2)^4\rtimes_\psi \mathcal{D}_3$ & 96 & $\langle 96,227\rangle$& 25 & 9 & a\\
$(\mathbb{Z}_2)^3\rtimes_\varphi \mathcal{D}_4$ & 64 & $\langle 64,73\rangle$& 9 & 17 & a\\
$U(4,2)$ & 64 & $\langle 64,138\rangle$& 9 & 17 & a\\
$\mathcal{A}_5$ & 60 & $\langle 60,5\rangle$& 13 & 11 & a\\
$\mathcal{A}_5$ & 60 & $\langle 60,5\rangle$& 41 & 4 & a\\
$\mathcal{A}_5$ & 60 & $\langle 60,5\rangle$& 9 & 16 & a\\
$\mathcal{A}_5$ & 60 & $\langle 60,5\rangle$& 5 & 31 & a\\
$\mathcal{S}_4\times\mathbb{Z}_2$ & 48 & $\langle 48,48\rangle$& 5 & 25& a \\
$\mathcal{S}_4\times\mathbb{Z}_2$ & 48 & $\langle 48,48\rangle$& 9 & 13 & a\\
$\mathcal{S}_4\times\mathbb{Z}_2$ & 48 & $\langle 48,48\rangle$& 13 & 9 & a\\
$\mathcal{S}_4\times\mathbb{Z}_2$ & 48 & $\langle 48,48\rangle$& 3 & 49 & a\\
$(\mathbb{Z}_2)^3\rtimes_\varphi\mathbb{Z}_4$ & 32 & $\langle 32,22\rangle$& 9 & 9 & \textbf{d}\\
$\mathcal{D}_4\times (\mathbb{Z}_2)^2$ & 32 & $\langle 32,46\rangle$& 9 & 9 & a\\
$(\mathbb{Z}_2)^4\rtimes_\varphi\mathbb{Z}_2$ & 32 & $\langle 32,27\rangle$& 17 & 5 & a\\
$(\mathbb{Z}_2)^4\rtimes_\varphi\mathbb{Z}_2$ & 32 & $\langle 32,27\rangle$& 9 & 9 & a\\
$\mathcal{S}_4$ & 24 & $\langle 24,12\rangle$& 5 & 13 & a\\
$\mathcal{S}_4$ & 24 & $\langle 24,12\rangle$& 3 & 25 & a\\
$\mathcal{D}_4\times\mathbb{Z}_2$ & 16 & $\langle 16,11\rangle$& 9 & 5 & a\\
$(\mathbb{Z}_2)^2\rtimes_\varphi\mathbb{Z}_4$& 16 & $\langle 16,3\rangle$& 9 & 5 & a \\
$(\mathbb{Z}_2)^4$ & 16 & $\langle 16,14\rangle$& 9 & 5 & a\\
$\mathcal{D}_4\times\mathbb{Z}_2$ & 16 & $\langle 16,11\rangle$& 3 & 17 & a\\
$(\mathbb{Z}_3)^2$ & 9 & $\langle 9,2\rangle$& 7 & 4 & \textbf{c}\\
$(\mathbb{Z}_2)^3$ & 8 & $\langle 8,5\rangle$& 5 & 5 & a\\
$(\mathbb{Z}_2)^3$ & 8 & $\langle 8,5\rangle$& 3 & 9 & a\\
\hline
\end{tabular}
\caption{\label{tabella} Complete list of groups that admit an unmixed ramification structure such that the corresponding surfaces $S$ isogenous to a higher product has $\chi(\mathcal{O_S})= 2$ and $q(S)=0$.  SGL is the pair that identifies the group in the Small Groups Library (on Magma).}
\end{center}
\end{table}

%%-->Section 3
\section{The exceptional cases}
In this section we study one by one the families of surfaces in Table \ref{tabella} not of type $a$.\\
Given a finite group $G$ and an unmixed ramification structure $(T_C,T_D)$ for $G$ we will use the following notation:
\begin{itemize}
\item $f:C\to\mathbb{P}^1$ and $h:D\to\mathbb{P}^1$ are the Galois covering associated to the spherical system of generators $T_C$ and $T_D$;
\item $S=\frac{C\times D}{G}$ is the surface isogenous to a product of unmixed type corresponding to the unmixed ramification structure;
\item $Z<H^2(S,\mathbb{Q})$ is the $4$-dimensional Hodge substructure with $\dim Z^{2,0}=1$ defined by 
\begin{equation*}
Z=\left(H^1(C,\mathbb{Q})\otimes H^1(D,\mathbb{Q})\right)^G.
\end{equation*}
\end{itemize}
Let $\tau:G\to GL(W)$ be an irreducible rational representation of $G$ and let $A_C,\,A_D$ be the isotypical components of $H^1(C,\mathbb{Q})$ and $H^1(D,\mathbb{Q})$ related to $\tau$.
Assume that $Z\cong(A_C\otimes A_D)^G$: as described in the proof of Theorem \ref{classification} this is exactly what happens for surfaces of type $a,\,b$ and $c$.\\
Let $H\triangleleft G$ be the normal subgroup $H=ker(\tau)$.  Then we get
\begin{equation}\label{coH}
\begin{split}
Z=&\left(H^1(C,\mathbb{Q})\otimes H^1(D,\mathbb{Q})\right)^G=\\
=&\left(H^1(C,\mathbb{Q})^H\otimes H^1(D,\mathbb{Q})^H\right)^{G/H}.
\end{split}
\end{equation}
\begin{remark}\label{rem}
Notice that, for a general subgroup $H\le G$, we have
\begin{equation*}
\left(H^1(C,\mathbb{Q})\otimes H^1(D,\mathbb{Q})\right)^H\not\cong\left(H^1(C,\mathbb{Q})^H\otimes H^1(D,\mathbb{Q})^H\right).
\end{equation*}
For example for $H=G$ we get the Hodge structure $Z$ on the left and the empty vector space on the right, since $C/G\cong D/G\cong\mathbb{P}^1$.
Equation \eqref{coH} holds because our specific choice of the subgroup $H$.
\end{remark}
Using this idea (with appropriate modifications for the case $d$) we extend the result of Theorem \ref{main} to the remaining surfaces.

%%--> Section 3.1
\subsection{Case b}\label{caseb}
Let $G$ be the finite group $G=G(128,36)$ with presentation:
\begin{equation*}
G=\left\langle g_1,\,...\,,g_7\quad\Big|\quad
\begin{array}{lll}
g_1^2=g_4 & g_2^2=g_5 & g_2^{g_1}=g_2g_3\\ 
g_3^{g_1}=g_3g_6 & g_3^{g_2}=g_3g_7 & g_4^{g_2}=g_4g_6\\
g_5^{g_1}=g_5g_7
\end{array}
\right\rangle,
\end{equation*}
where $g_i^{g_j}:=g_j^{-1}g_ig_j$; $G$ has order $128$ and it determined by the pair $\langle 128,36\rangle$ in the Small Groups Library on Magma.  Consider the unmixed ramification structure $(T_C,\,T_D)$ of type $([4^3],[4^3])$:
\begin{equation*}
\begin{array}{l}
T_C=[g_1g_2g_4g_6, g_1g_4g_5g_6,g_2g_3g_4g_7],\\ T_D=[g_1g_2g_3g_6g_7,g_2g_5g_7,g_1g_3g_4g_7].
\end{array}
\end{equation*}
By direct computation we verify that the corresponding surface isogenous to a product $S$ is of type $b$, i.e.\ there exists an irreducible rational representation $\tau:G\to GL(W)$, $\dim W=4$ and an irreducible complex representation $\rho:G\to GL(V)$, $\dim V=2$ with $\tau_\mathbb{C}=2\rho$ such that
\begin{equation*}
\begin{split}
&n_C(\tau)=n_D(\tau)=1,\\
&n_C(\tau_j)\cdot n_D(\tau_j)=0\quad\mbox{$\forall\tau_j$ different from $\tau$}.
\end{split} 
\end{equation*}
Let $H\triangleleft G$ be the normal subgroup $H:=Ker(\tau)$: a set of generators for $H$ is
\begin{equation*}
H=\langle g_7,g_6,g_3g_4,g_4g_5\rangle.
\end{equation*}
The quotient group $G/H$ has order $8$ and it is isomorphic to the quaternion group $Q_8$.
Consider the intermediate coverings:
\begin{equation*}
\xymatrix{
C \ar[r]^H \ar[dr]_{G}
& C' \ar[d]^{Q_8}\\
& \mathbb{P}^1}
\qquad
\xymatrix{
D \ar[r]^H \ar[dr]_{G}
& D' \ar[d]^{Q_8}\\
& \mathbb{P}^1}
\end{equation*}
The curves $C'$ and $D'$ have genus $2$, by Riemann-Hurwitz formula.  Moreover the quaternion group $Q_8$ acts on their rational cohomology by the rational representation of dimension $4$ described in Example \ref{quaternion}.
By the Remark \ref{rem} we get
\begin{equation*}
H^2(S,\mathbb{Q})\cong H^2\left(C'\times D',\mathbb{Q}\right)^{Q_8}.
\end{equation*}
\begin{proposition}
Let $S$ be the surface isogenous to a higher product defined above.  Then $H^2(S,\mathbb{Q})\cong H^2(E_{\sqrt{-2}}\times E_{\sqrt{-2}},\mathbb{Q})$ where $E_{\sqrt{-2}}$ is the elliptic curve
\begin{equation*}
E_{\sqrt{-2}}=\frac{\mathbb{C}}{\mathbb{Z}\oplus\sqrt{-2}\mathbb{Z}}.
\end{equation*}
\end{proposition}
\begin{proof}
Let $X$ be a curve of genus $2$ such that $Q_8\le Aut(X)$ and $X/Q_8\cong\mathbb{P}^1$.  Then its Jacobian is not simple, and in particular it is isogneous to the self-product of the elliptic curve $E_{\sqrt{-2}}$.
The action of $Q_8$ induces a Hodge morphism $\psi$
\begin{equation*}
\psi: H^1(E_{\sqrt{-2}},\mathbb{Q})\otimes H^1(E_{\sqrt{-2}},\mathbb{Q})\to Z.
\end{equation*}
Now, arguing as in the step $2$ of the proof of the Theorem \ref{main}, we conclude that
\begin{equation*}
H^2(S,\mathbb{Q})\cong H^2\left(C'\times D',\mathbb{Q}\right)^{Q_8}\cong H^2(E_{\sqrt{-2}}\times E_{\sqrt{-2}},\mathbb{Q}).
\end{equation*}
\end{proof}
\begin{remark}
The covering maps $f:C\to\mathbb{P}^1$ and $h:D\to\mathbb{P}^1$ have both $3$ branching values.  It follows that the curves $C$ and $D$ are determined up to isomorphism, by the Riemann Existence Theorem.  In particular the pairs $(C,\,f)$ and $(D,\,g)$ are Belyi pairs and the surface $S$ is a Beauville surface.
\end{remark} 

%%--> Section 3.2
\subsection{Case c}\label{casec}
Two groups occur in this case.  Let $G$ be the finite group $(\mathbb{Z}_3)^2$ and consider the unmixed ramification structure $(T_C,T_D)$:
\begin{equation*}
\begin{split}
T_C=&[(1,1),(2,1),(1,1),(1,2),(1,1)];\\
T_D=&[(0,2),(0,1),(1,0),(2,0)].
\end{split}
\end{equation*}
This structure has been already studied in Example \ref{z32}: notice that the corresponding surface isogenous to a product $S$ is of type $c$.\\
In particular, using the notation of Example \ref{z32}, there is an irreducible rational representation $\tau_4:G\to GL(W_4)$ such that
\begin{itemize}
\item $\tau_4\otimes\mathbb{C}=\rho_5\oplus\rho_9$;
\item $n_C(\tau_4)=1$ and $n_D(\tau_4)=2$.
\end{itemize}
Let $H$ be the normal subgroup $H:=Ker(\rho_5)=Ker(\rho_9)$: a set of generators for $H$ is
\begin{equation*}
H=\langle(2,1)\rangle.
\end{equation*} 
Notice that $H\cong\mathbb{Z}_3$ and also $G/H\cong\mathbb{Z}_3$.
Let us consider the intermediate coverings $C'=C/H$ and $D' =D/H$ of genus $g(C')=1$ and $g(D')=2$.\\
The curve $D'$ is a curve of genus $2$ with an automorphism $\sigma$ of order $3$ such that $D'/\langle\sigma\rangle\simeq\mathbb{P}^1$.  It follows that its Jacobian is not simple and in particular it is isogenous to the self-product of an elliptic curve $E_D$.

\begin{proposition}
Let $S$ be the regular surface isogenous to a product of unmixed type associated to the unmixed structure $(T_C,T_D)$.  Then $H^2(S,\mathbb{Q})\simeq H^2(C'\times E_D,\mathbb{Q})$, where $C'$ and $E_D$ are the elliptic curves described above.
\end{proposition}
\begin{proof}
By Remark \ref{rem} we have $H^2(S,\mathbb{Q})\cong H^2(C'\times D',\mathbb{Q})^G$ and we have already observed that the cohomology group $H^1(D',\mathbb{Q})$ decomposes as sum of two Hodge substructures, both of dimension $2$.  Now we conclude with the same arguments used in the proof of Theorem \ref{main}.
\end{proof}

The case of the group $G=(\mathbb{Z}_2)^4\rtimes_\varphi \mathcal{D}_3$ follows in a similar way.  This group has $14$ irreducible complex representations with Schur index $1$: $12$ are self-dual while the remaining two are in the same Galois-orbit.  
So we have an irreducible rational representation $\tau:G\to GL(W)$ such that $\tau\otimes\mathbb{C}$ decompose as sum of two irreducible complex representations.  We set $H=Ker(\tau)$ and we proceed as before.

%%--> Section 3.3
\subsection{Case d}\label{cased}
Let $G$ be the group $G=(\mathbb{Z}_2)^3\rtimes_\varphi\mathbb{Z}_4$ where $\varphi:\mathbb{Z}_4\to Aut(\mathbb{Z}_2^3)\simeq GL(3,\mathbb{F}_2)$ is defined by
\begin{equation*}
\varphi(1)=
\begin{pmatrix}
1&0&0\\
0&1&0\\
1&0&1
\end{pmatrix}.
\end{equation*}
Consider the unmixed ramification structure $(T_C,T_D)$ of $G$ of type $([2^2,4^2],[2^2,4^2])$:
\begin{equation*}
\begin{split}
T_C=[((1,0,0),2),((1,1,1),2),((0,1,0),1),((0,0,1),3)],\\
T_D=[((1,1,0),0),((1,0,0),0),((1,0,0),3),((1,1,1),1)].
\end{split}
\end{equation*}
We construct the group $G$ in \cite{magma}:
\begin{verbatim}
H:=CyclicGroup(4);
K:=SmallGroup(8,5);
A:=AutomorphismGroup(K);
M:=hom<K->K|[K.1->K.1*K.3, K.2->K.2, K.3->K.3]>;
Phi:=hom<H->A|[H.1->M]>;
G,a,b:=SemidirectProduct(K,H,Phi);
G1:=a(K.1);
G2:=a(K.2);
G3:=a(K.3);
G4:=b(H.1);
\end{verbatim}
With this notation the unmixed ramification structure is given by:
\begin{equation*}
T_C=[g_1g_4^2,\, g_1g_2g_3g_4^2,\,g_2g_4,\,g_3g_4^3], \quad T_D=[g_1g_2,\,g_1,\,g_1g_4^3,\,g_1g_2g_3g_4].
\end{equation*}
By direct calculation we see that the surface $S$ is of type $d$.
We denote by $\tau_{j_1}:G \to GL(W_{j_1})$, $\tau_{j_2}:G\to GL(W_{j_2})$ the two irreducible rational representations such that
\begin{equation*}
\begin{split}
&n_C(\tau_{j_1})=n_C(\tau_{j_2})=n_D(\tau_{j_1})=n_D(\tau_{j_2})=1,\\
&n_C(\tau_j)\cdot n_D(\tau_j)=0,\quad \forall j\mbox{ different from }j_1,\,j_2.
\end{split}
\end{equation*}
We set $H_1:=ker(\tau_{j_1})$ and $H_2:=ker(\tau_{j_2})$ of $G$.  A set of generators for $H_1$ and $H_2$ are
\begin{equation*}
\begin{split}
H_1&=\langle((1,0,0),0),((0,0,1),0),((0,1,0),2)\rangle=\langle g_1,\,g_3,\,g_2g_4^2\rangle,\\
H_2&=\langle((1,1,0),0),((0,0,1),0),((0,1,0),2)\rangle=\langle g_1g_2,\,g_3,\,g_2g_4^2 \rangle.
\end{split}
\end{equation*}
We observe that:
\begin{itemize}
\item $G/H_1\cong G/H_2\cong \mathbb{Z}_4$;
\item the curves $C_1:=C/H_1$, $C_2:=C/H_2$, $D_1:=D/H_1$ and $D_2:=D/H_2$ have genus $1$.
\end{itemize}
Consider the intermediate coverings:
\begin{equation*}
\xymatrix{
C \ar[r]^{H_i} \ar[dr]_{G}
& C_i\ar[d]^{\mathbb{Z}_4}\\
& \mathbb{P}^1}
\qquad
\xymatrix{
D \ar[r]^{H_i} \ar[dr]_{G}
& D_i \ar[d]^{\mathbb{Z}_4}\\
& \mathbb{P}^1}
\end{equation*}
Since $C_i$ and $D_i$, $i=1,\,2$ are elliptic curves with an automorphism of order $4$ they are all isogenous to
\begin{equation*}
E_i=\frac{\mathbb{C}}{\mathbb{Z}\oplus i\mathbb{Z}}.
\end{equation*}
By Remark \ref{rem} we get
\begin{equation*}
Z=\left(H^1(C_1,\mathbb{Q})\otimes H^1(D_1,\mathbb{Q})\right)^G\oplus \left(H^1(C_2,\mathbb{Q})\otimes H^1(D_2,\mathbb{Q})\right)^G.
\end{equation*}
\begin{proposition}
Let $S$ be the surface isogenous to a higher product defined above.  Then $H^2(S,\mathbb{Q})=H^2(E_i\times E_i,\mathbb{Q})$, as rational Hodge structures.
\end{proposition}
\begin{proof}
We have already observed that
\begin{equation*}
Z=\left(H^1(C_1,\mathbb{Q})\otimes H^1(D_1,\mathbb{Q})\right)^G\oplus \left(H^1(C_2,\mathbb{Q})\otimes H^1(D_2,\mathbb{Q})\right)^G.
\end{equation*}
Up to exchange of $C_1\times D_1$ with $C_2\times D_2$, we can assume that the Hodge structure $W:=(H^1(C_1,\mathbb{Q})\otimes H^1(D_1,\mathbb{Q}))^G$ has dimension $2$ and $\dim W^{2,0}=\dim W^{0,2}=1$.  Now following the same idea of the proof of Theorem \ref{main} we get:
\begin{equation*}
H^2(S,\mathbb{Q})\cong H^2(C_1\times D_1,\mathbb{Q})\cong H^2(E_i\times E_i,\mathbb{Q}).
\end{equation*}
\end{proof}

\begin{remark} 
Consider, as in the proof of Theorem \ref{main}, the Hodge morphism $\psi:H^1(C_1,\mathbb{Q})\otimes H^1(D_1,\mathbb{Q})\to Z$.  Here it is clear that $\psi$ is not an isomorphism since its image $Im(\psi)$ has dimension $2$.
\end{remark}

%%--> Section 4
\section{Conclusion}
\begin{theorem}\label{main+}
Let $S$ be a regular surface isogenous to a higher product of unmixed type with $\chi(\mathcal{O}_S)=2$.  Then there exist two elliptic curves $E_C$ and $E_D$ such that $H^2(S,\mathbb{Q})\cong H^2(E_C\times E_D, \mathbb{Q})$ as rational Hodge structures.
\end{theorem}
\begin{proof}
It follows from Theorem \ref{main} and the analysis, case by case, of the previous section.
\end{proof}

\begin{remark}
In general the Theorem does not imply the existence of intermediate covering of the curves $C$, $D$.
More precisely there are no subgroups $H_C,\,H_D$ of $G$ such that $C/H_C\cong E_C$, $C/H_D\cong E_D$ where $E_C,\,E_D$ are elliptic curves such that $H^2(S,\mathbb{Q})\cong H^2(E_C\times E_D, \mathbb{Q})$.  See the following example.
\end{remark}

\begin{example}
Consider once more the unmixed ramification structure studied in Example \ref{z32} and in Section \ref{casec}.  Let $G$ be the abelian group $(\mathbb{Z}_3)^2$ and let $T_D$ be the spherical system of generators
\begin{equation*}
T_D=[(0,2),(0,1),(1,0),(2,0)].
\end{equation*}
such that the corresponding curve $D$ has genus $4$.  Consider all the $6$ subgroups of $G$.  By \cite{magma} we verify that for all subgroups $H\le G$ the quotient curve $D/H$ has genus $0,\,2$ or $4$.
In particular there is not any subgroup $H$ such that $D/H$ is an elliptic curve.
\end{example}

%%-->Section 4.1
\subsection{About the Picard number}\label{picard}

Let $S$ be a regular surfaces isogenous to a higher product of unmixed type with $\chi(\mathcal{O}_S)=2$.  We can compute $\rho (S)$, the Picard number of $S$, using Theorem \ref{main+}.  Let $E_C$ and $E_D$ be the elliptic curves such that $H^2(E_C\times E_D,\mathbb{Q})\cong H^2(S,\mathbb{Q})$: we get $\rho(S)=\rho(E_C\times E_D)$.
The Picard number of an Abelian surface of product type $E_1\times E_2$ is
\begin{equation*}
\rho(E_1\times E_2)=
\begin{cases}
4 & \mbox{if }E_1\sim E_2\mbox{ has complex multiplication},\\
3 & \mbox{if }E_1\sim E_2\mbox{ but they do not have CM},\\
2 & \mbox{otherwise}.\end{cases}
\end{equation*}

A surface $S$ is said to be a surface with maximal Picard number if $\rho(S)=h^{1,1}(S)$: this kind of surfaces are studied in a recent work of Beauville \cite{Be13} where a lot of examples are constructed.
As already observed a regular surface $S$ isogenous to a higher product of unmixed type with $\chi(\mathcal{O}_S)=2$ has $h^{1,1}(S)=4$.  It follows that the surfaces studied in Sections $\ref{caseb}$ and $\ref{cased}$ are examples of surfaces with maximal Picard number.

% Bibliograph

\end{document}